\documentclass[11pt]{article}
\usepackage{mathbbold}
\usepackage{amsfonts}
\usepackage{mathrsfs}
\usepackage{bbm}
\usepackage{latexsym,amsfonts,amssymb,amsmath,amsthm}
\usepackage{graphicx}

\usepackage{setspace}

\newtheorem{theorem}{Theorem}[section]
\newtheorem{definition}[theorem]{Definition}
\newtheorem{lemma}[theorem]{Lemma}

\newtheorem{prop}[theorem]{Proposition}
\numberwithin{equation}{section}
\def\u{\underline}

\begin{document}

\baselineskip 20pt

\begin{center}

\textbf{\Large A random attractor for stochastic porous media
equations on infinite lattices}

\vskip 0.5cm

{\large  Anhui Gu, Yangrong Li, Jia Li}

\vskip 0.3cm

\textit{School of Mathematics and Statistics, Southwest
University, Chongqing, 400715, China}\\

\vskip 1cm

\begin{minipage}[c]{15cm}

\noindent \textbf{Abstract}: The paper is devoted to studying the
existence of a random attractor for stochastic porous media
equations on infinite lattices under some conditions. \vspace{5pt}

\vspace{5pt}


\end{minipage}
\end{center}

\vspace{10pt}

\baselineskip 18pt

\section{Introduction}\label{s1}
In this paper, we study the following stochastic porous media
lattice equations perturbed by a multiplicative noise:
\begin{eqnarray} \label{sys1}
\begin{split}
\frac{du_i(t)}{dt}&=\Phi(u_{i-1})-2\Phi(u_i)+\Phi(u_{i+1})\\
&\quad\quad -\lambda
(u_i+|u_i|^{p-1}u_i)+g_i+\alpha_iu_i\circ\frac{dw(t)}{dt}
\end{split}
\end{eqnarray}
with the initial data
\begin{eqnarray} \label{initial}
\begin{split}
u_i(0)=u_{0,i}, \quad i\in \mathbb{Z},
\end{split}
\end{eqnarray}
where $\mathbb{Z}$ denotes the integer set, $u_i\in \mathbb{R}$,
$\lambda$ and $\alpha_i\in \mathbb{R}$ are positive constants,
$p>1$, $\Phi$ satisfies certain dissipative conditions, $g_i\in
\mathbb{R}$, $w(t)$ is a Brownian motion (Wiener process) and
$\circ$ denotes the Stratonovich sense of the stochastic term.

System \eqref{sys1} can be regarded as the spatial discrete form on
1$\mathbf{D}$ infinite lattices of a type of stochastic porous media
equations
\begin{align} \label{system2}
u_t=&\Delta(\Phi(u))-\lambda (u+|u|^{p-1}u)+g(x)+\alpha u\circ
\frac{dw(t)}{dt}, \quad x\in \mathbb{R}.
\end{align}
When $\lambda=0$, the (similar) stochastic porous media equations
have been intensively investigated in recent years, see e.g.
\cite{BDR1, BDR2, BDR3, DRRW, Kim, RW} and references therein. The
long-time behavior of stochastic porous media equations with
additive white noise in terms of the existence of a random attractor
has first been established in \cite{BGLR} which then has been
extended to more generally distributed additive noise in
\cite{Gess1, Gess2}, and linear multiplicative noise in space and
time in \cite{Gess4}.

Recently, lattice dynamical systems, which can be considered as the
\textit{spatial discrete} of some PDEs, have drawn much attention
from mathematicians and physicists, due to the wide range of
applications in various areas (see \cite{Chow}). For existence and
properties of different attractors for autonomous deterministic
lattice dynamical systems, see e.g. \cite{BLW, KY, Wang1, Zh4, Zh5}
and e.g. \cite{Wang2, ZH} for non-autonomous deterministic cases.
For stochastic ones, stochastic lattice dynamical systems (SLDS)
arise naturally while random influences or uncertainties are taken
into account, these noises may play an important role as intrinsic
phenomena rather than just compensation of defects in deterministic
models. Since Bates et al. \cite{BLL} initiated the study of SLDS, a
lot of work has been done regarding the existence of global random
attractors for SLDS with white noise in regular or weighted spaces
of infinite sequences, see e.g. \cite{CL, CMV, HSZ}. For lattice
dynamical systems perturbed by ``rough" noises, see \cite{Gu1, Gu2,
Gu3} for more details.

Notice that there are amounts of work considered for random
attractors for PDEs defined on \textit{unbounded domains} (see e.g.
\cite{BLW2, BL, Wang3, Wang4}). This introduces a major obstacle
that Sobolev embeddings are not compact for these cases. In fact,
the study of lattice differential systems generated by some
\textit{spatially discrete} PDEs can be regarded as the cases
considered on \textit{unbounded domains}. We take the
square-summable infinite sequences space $\ell^2$ as the phase
space, so that there is not any embedding relationships to tackle
these questions. Often, lattice models are used more in a physical
and numerical treatment in porous media (see e.g. \cite{JZ, KLJ,
SP}). There are few detailed analysis in the sense of infinite
dynamical systems. Furthermore, we can see that all the references
on the asymptotic behavior of stochastic porous media equations
studied above are restricted to the \textit{bounded domains}. To our
knowledge, there is no result in the case of \textit{unbounded
domains}. Here, we set up the stochastic porous media lattice
equations, and give the existence of a random attractor for the
lattice model, which can be seen as a first attempt to the
\textit{unbounded} cases.

The paper is organized as follows. In next section, we recall some
preliminaries on random dynamical systems and random attractors. In
section \ref{s3}, we formulate the model of stochastic porous media
lattice equations and give a unique solution to system \eqref{sys1},
which generates a continuous random dynamical system. We obtain the
main result, that is the existence of a random attractor, in section
\ref{s4}.

In the sequel, we denote $\ell^p$ the space of $p$-times summable
infinite sequences with norm $\|\cdot\|_p$, especially when $p=2$,
we denote $\ell^2=(\ell^2, (\cdot, \cdot), \|\cdot\|)$.

\section{Preliminaries}\label{s2}

For the reader's convenience, we first introduce some basic concepts
related to random dynamical systems and random attractors, which are
taken from \cite{Arnold},  \cite{Chueshov} and \cite{HSZ}. Let
$(\mathbb{H}, \|\cdot\|_{\mathbb{H}})$ be a separable Hilbert space
and $(\Omega, \mathcal{F}, \mathbb{P})$ be a probability space.

\begin{definition}
A stochastic process $\{\varphi(t, \omega)\}_{t\geq 0, \omega\in
\Omega}$ is a continuous random dynamical system (RDS) over
$(\Omega, \mathcal{F}, \mathbb{P},(\theta_{t})_{t\in \mathbb{R}})$
if $\varphi$ is $(\mathcal{B}[0,\infty)\times \mathcal{F}\times
\mathcal{B}(\mathbb{H}), \mathcal{B}(\mathbb{H}))$-measurable, and
for all $\omega \in \Omega$,

(i) the mapping $\varphi(t,\omega): \mathbb{H}\mapsto \mathbb{H}$,
$x\mapsto \varphi(t,\omega)x$ is continuous for every $t\geq 0$,

(ii)  $\varphi(0,\omega)$ is the identity on $\mathbb{H}$,

(iii) \mbox{(cocycle property)} \
$\varphi(s+t,\omega)=\varphi(t,\theta_{s}\omega)\varphi(s,\omega)$
for all $s, t\geq 0$.
\end{definition}

\begin{definition}  \label{tempered random set}
(i) A set-valued mapping $\omega\mapsto B(\omega)\subset \mathbb{H}$
(we may write it as $B(\omega)$ for short) is said to be a random
set if the mapping $\omega\mapsto$ dist$_{\mathbb{H}}(x, B(\omega))$
is measurable for any $x\in \mathbb{H}$, where dist$_{\mathbb{H}}(x,
D)$ is the distance in $\mathbb{H}$ between the element $x$ and the
set $D\subset \mathbb{H}$.

(ii) A random set $B(\omega)$ is said to be bounded if there exist
$x_0\in \mathbb{H}$ and a random variable $r(\omega)>0$ such that
$B(\omega)\subset\{x\in \mathbb{H}: \|x-x_0\|_{\mathbb{H}}\leq
r(\omega), x_0\in \mathbb{H}\}$ for all $\omega \in \Omega$.

(iii) A random set $B(\omega)$ is called a compact random set if
$B(\omega)$ is compact for all $\omega \in \Omega$.

(iv) A random bounded set $B(\omega) \subset \mathbb{H}$ is called
tempered with respect to $(\theta_{t})_{t\in \mathbb{R}}$ if for
a.e. $\omega \in \Omega$, \ $\lim_{t\rightarrow +\infty}e^{-\gamma
t}d(B(\theta_{-t}\omega))=0 \ \ \mbox{for all} \ \ \gamma > 0$,
where $d(B)=\sup_{x\in B}\|x\|_{\mathbb{H}}$. A random variable
$\omega \mapsto r(\omega)\in \mathbb{R}$ is said to be tempered with
respect to $(\theta_{t})_{t\in \mathbb{R}}$ if for a.e. $\omega \in
\Omega$, $\lim_{t\rightarrow +\infty} \sup_{t\in
\mathbb{R}}e^{-\gamma t}r(\theta_{-t}\omega)=0 \ \ \mbox{for all} \
\ \gamma > 0$.

\end{definition}

We consider a RDS $\{\varphi(t, \omega)\}_{t\geq 0, \omega\in
\Omega}$ over  $(\Omega, \mathcal{F}, \mathbb{P},(\theta_{t})_{t\in
\mathbb{R}})$ and $\mathcal{D}(\mathbb{H})$ the set of all tempered
random sets of $\mathbb{H}$.

\begin{definition}
A random set $\mathcal{K}$ is called an absorbing set in
$\mathcal{D}(\mathbb{H})$ if for all $B\in \mathcal{D}(\mathbb{H})$
and a.e. $\omega \in \Omega$ there exists $t_{B}(\omega)>0$ such
that
$$\varphi(t,\theta_{-t}\omega)B(\theta_{-t}\omega)\subset
\mathcal{K}(\omega) \ \ \mbox{for all} \ \ t\geq t_{B}(\omega).$$

\end{definition}

\begin{definition}
A random set $\mathcal{A}$ is called a global random
$\mathcal{D}(\mathbb{H})$ attractor (pullback
$\mathcal{D}(\mathbb{H})$ attractor) for $\{\varphi(t,
\omega)\}_{t\geq 0, \omega\in \Omega}$ if the following hold:

(i) $\mathcal{A}$ is a random compact set, i.e. $\omega\mapsto d(x,
\mathcal{A}(\omega))$ is measurable for every $x\in \mathbb{H}$ and
$\mathcal{A}(\omega)$ is compact for a.e. $\omega \in \Omega$;

(ii)  $\mathcal{A}$ is strictly invariant, i.e. for $\omega \in
\Omega$ and all $t\geq 0$,
$\varphi(t,\omega)\mathcal{A}(\omega)=\mathcal{A}(\theta_{t}\omega)$;

(iii)  $\mathcal{A}$ attracts all sets in $\mathcal{D}(\mathbb{H})$,
i.e. for all $B\in \mathcal{D}(\mathbb{H})$ and a.e. $\omega \in
\Omega$,
$$\lim_{t\rightarrow+\infty}dist(\varphi(t,\theta_{-t}\omega)
B(\theta_{-t}\omega), \mathcal{A}(\omega))=0,$$ where
$dist(X,Y)=\sup_{x\in X} \inf_{y\in Y}\|x-y\|_{\mathbb{H}}$ is the
Hausdorff semi-metric ($X\subseteq \mathbb{H}, Y\subseteq
\mathbb{H}$).\end{definition}

\begin{prop}(See \cite{BLL}.) \label{attractor1}
Let $K\in \mathcal{D}(\mathbb{H})$ be an absorbing set for the
continuous RDS $\{\varphi(t, \omega)\}_{t\geq 0, \omega\in \Omega}$
which is closed and which satisfies for a.e. $\omega\in \Omega$ the
following asymptotic compactness condition: each sequence $x_n\in
\varphi(t_n, \theta_{-t_n}, K(\theta_{-t_n}\omega)$ with
$t_n\rightarrow\infty$  has a convergent subsequence in
$\mathbb{H}$. Then the cocycle $\varphi$ has a unique global random
attractor
\begin{eqnarray*}
\mathcal{A}(\omega)=\bigcap_{\tau\geq T_K(\omega)}
\overline{\bigcup_{t\geq \tau} \varphi(t,\theta_{-t}\omega,
K(\theta_{-t}\omega))}.
\end{eqnarray*}
\end{prop}

Especially, when we focus on SLDS and denote $\mathcal{D}(\ell^2)$
the set of all tempered random sets of $\ell^2$, it yields the
following result:

\begin{prop}(See \cite{HSZ}.) \label{condition} Suppose that

(a) there exists a random bounded absorbing set
$\mathcal{K}(\omega)\in \mathcal{D}(\ell^2)$ for the continuous RDS
$\{\varphi(t, \omega)\}_{t\geq 0, \omega\in \Omega}$;

(b) the RDS $\{\varphi(t, \omega)\}_{t\geq 0, \omega\in \Omega}$ is
random asymptotically null on $\mathcal{K}(\omega)$, i.e., for any
$\epsilon>0$, there exist $T(\epsilon, \omega, \mathcal{K})>0$ and
$I_0(\epsilon, \omega, \mathcal{K})\in \mathbb{N}$ such that
\begin{eqnarray}\label{Null}
\sup_{u\in \mathcal{K}(\omega)}\sum_{|i|>I_0(\epsilon, \omega,
\mathcal{K}(\omega))}| (\varphi_i(t, \theta_{-t}\omega,
u(\theta_{-t}\omega))|^2\leq \epsilon^2, \ \ \forall t\geq
T(\epsilon, \omega, \mathcal{K}(\omega)).
\end{eqnarray}

Then the RDS $\{\varphi(t, \omega, \cdot)\}_{t\geq 0, \omega\in
\Omega}$ possesses a unique global random $\mathcal{D}(\ell^2)$
attractor given by
\begin{eqnarray*}
\mathcal{\tilde{A}}(\omega)=\bigcap_{\tau\geq T(\omega,
\mathcal{K})} \overline{\bigcup_{t\geq \tau}
\varphi(t,\theta_{-t}\omega, \mathcal{K}(\theta_{-t}\omega))}.
\end{eqnarray*}
\end{prop}

\section{Stochastic porous media lattice equations}\label{s3}
We note that system \eqref{sys1} can be interpreted as a system of
integral equations
\begin{eqnarray} \label{sys2}
\begin{split}
u_i&(t)=u_i(0)+\int_0^t[\Phi(u_{i-1}(s))-2\Phi(u_i(s))
+\Phi(u_{i+1}(s))\\
&-\lambda
(u_i(s)+|u_i(s)|^{p-1}u_i(s))+g_i]ds+\alpha_i\int_0^tu_i(s)\circ
dw(s), \quad i\in \mathbb{Z},
\end{split}
\end{eqnarray}
where the stochastic integral is understood to be in the
Stratonovich sense.

{\bf Assumptions on $\Phi$} \quad $\Phi: \mathbb{R}\rightarrow
\mathbb{R}$ is continuous, $\Phi(0)=0$ and there exist constants
$c_1, c_2\in (0, \infty)$ such that
\begin{eqnarray}\label{cond0}
\frac{(p+1)^2}{4}c_2|u|^{p-1}\le\Phi'(u)\le c_1(1+|u|^{p-1}), \
\forall ~u\in \mathbb{R}.
\end{eqnarray}
Then \eqref{cond0} implies the following \textit{monotonicity}
condition:
\begin{eqnarray}
(\Phi(u)-\Phi(v))(u-v)\ge k|u-v|^{p+1}-a_i, \ \forall ~u, v\in
\mathbb{R},\label{cond1}
\end{eqnarray}
where $k\in (0, \infty)$ and $a_i\in [0, \infty)$ such that
$(a_i)_{i\in \mathbb{Z}}\in \ell^1$, $p>1$ the same one in
\eqref{sys1}. Actually, these conditions are satisfied for
$\Phi(u)=u|u|^{p-1}$ (see e.g. \cite{BGLR, DRRW}).

For convenience, we now formulate system \eqref{sys2} as a
stochastic differential equation in $\ell^2$. Define $\mathbf{B}$
and its adjoint operator $\mathbf{B}^*$ from $\ell^2$ to $\ell^2$ as
follows. For $u=(u_i)_{i\in \mathbb{Z}}\in \ell^2$,
\begin{eqnarray*} \label{operator1}
(\mathbf{B}u)_i=u_{i+1}-u_i, ~ (\mathbf{B}^*u)_i=u_{i-1}-u_i.
\end{eqnarray*}
Also, define the operator $ (Au)_i=-u_{i-1}+2u_i-u_{i+1}.$ We have
$(Au, u)=(\mathbf{B}u, \mathbf{B}^*u)=\|\mathbf{B}u\|^2\le
4\|u\|^2$, which means that $A$ is a bounded operator from $\ell^2$
to itself.

Now, system \eqref{sys1} with initial values $u_0=(u_{0,i})_{i\in
\mathbb{Z}}$ can be rewritten as the following equation in $\ell^2$
for $t\ge 0$ and $\omega\in \Omega$,
\begin{eqnarray} \label{sys3}
\begin{split}
u(t)=u_0+\int_0^t&[-A(\Phi (u(s)))-\lambda
(u(s)+|u(s)|^{p-1}u(s))+g]ds\\
&+\alpha\int_0^tu(s)\circ dw(s),
\end{split}
\end{eqnarray}
where $A(\Phi (u))=A(\Phi (u_i)_{i\in
\mathbb{Z}})=(-\Phi(u_{i-1})+2\Phi(u_i) -\Phi(u_{i+1}))_{i\in
\mathbb{Z}}, g=(g_{i})_{i\in \mathbb{Z}}\in \ell^2$ and
$\alpha=(\alpha_{i})_{i\in \mathbb{Z}}$.

To prove that this stochastic equation \eqref{sys3} generates a
random dynamical system, we will transform it into a random
differential equation in $\ell^2$. First, we need to recall some
properties of the Ornstein-Uhlenbeck processes.

Consider $(\Omega, \mathcal{F}, \mathbb{P})$ be a probability space,
where $\Omega$ is a subset of $\mathcal{C}_0(\mathbb{R},
\mathbb{R})=\{\omega\in \mathcal{C}(\mathbb{R}, \mathbb{R}):
\omega(0)=0\},$ endowed with the compact open topology (see
\cite{Arnold}), $\mathcal{F}$ is the Borel $\sigma$-algebra and
$\mathbb{P}$ is the corresponding Wiener measure on $\Omega$. Let
$\theta_t\omega(\cdot)=\omega(\cdot+t)-\omega(t)$, $t\in
\mathbb{R}$, then $(\Omega, \mathcal{F},
\mathbb{P},(\theta_{t})_{t\in \mathbb{R}})$ is an ergodic metric
dynamical system.

To solve \eqref{sys3}, we make a change of variables
\begin{eqnarray*} \label{Solution}
v(t)=e^{-\alpha z(\theta_t\omega)}u(t),
\end{eqnarray*}
where $u(t)$ is a solution of \eqref{sys3} and
$z(\theta_t\omega)=-\int^0_{-\infty}e^{\tau}\theta_t\omega(\tau)d\tau$
is a pathwise solution of the Ornstein-Uhlenbeck equation
\begin{eqnarray} \label{OU1}
dz+zdt=dw(t).
\end{eqnarray}
By \cite{BLL, CL}, we know that there exists a $\theta_t$-variant
set $\Omega'\subseteq \Omega$ of full $\mathbb{P}$ measure such that
$z(\theta_t\omega)$ is continuous in $t$ for every $\omega\in
\Omega'$, and the random variable $|z(\omega)|$ is tempered. In
addition, for every $\omega\in \Omega'$, we have the following
limits:
\begin{eqnarray}\label{temperness}
\begin{split}
\lim_{t\rightarrow\pm\infty}\frac{|z(\theta_t\omega)|}{|t|}=0
~\mbox{and}
~\lim_{t\rightarrow\pm\infty}\frac{1}{t}\int_0^tz(\theta_s\omega)ds=0.
\end{split}
\end{eqnarray}
Hereafter, we will write $\Omega$ as $\Omega'$ instead. Then $v(t)$
satisfies the following evolution equation with random coefficients
but without white noise
\begin{eqnarray} \label{sys4}
\begin{split}
\frac{dv(t)}{dt}=-e^{-\alpha z(\theta_t\omega)}A(\Phi& (e^{\alpha
z(\theta_t\omega)}v))+(\alpha
z(\theta_t\omega)-\lambda)v\\
&-\lambda e^{\alpha(p-1) z(\theta_t\omega)}|v|^{p-1}v+e^{-\alpha
z(\theta_t\omega)}g, \\
v(0)&=v_0\in \ell^2.
\end{split}
\end{eqnarray}

Now, we have the following result.
\begin{theorem}\label{uniquness}
Let $T>0$ and $v_0\in \ell^2$. Then the following two statements
hold:

(1) for every $\omega\in \Omega$, system \eqref{sys4} has a unique
solution $v(\cdot, \omega, v_0)\in \mathcal{C}([0, T], \ell^2)$;

(2) the solution $v$ of \eqref{sys4} depends continuously on the
initial data $v_0$.
\end{theorem}

\begin{proof}
(1) For any fixed $T>0$ and $\omega\in \Omega$, let $u, v\in
\ell^2$,
\begin{eqnarray*}
\begin{split}
\|A(\Phi &(e^{\alpha z(\theta_t\omega)}u))-A(\Phi (e^{\alpha
z(\theta_t\omega)}v))\|^2\\
&=\sum_{i\in \mathbb{Z}}(A(\Phi (e^{\alpha
z(\theta_t\omega)}u_i))-A(\Phi (e^{\alpha
z(\theta_t\omega)}v_i)))^2\\
&\quad\le 12 \sum_{i\in \mathbb{Z}}(\Phi (e^{\alpha
z(\theta_t\omega)}u_i)-\Phi (e^{\alpha
z(\theta_t\omega)}v_i))^2\\
&\quad\quad\le 12c^2_1(1+e^{\alpha(p-1)
z(\theta_t\omega)}(\|u\|+\|v\|)^{p-1})^{2}\|u-v\|^{2},
\end{split}
\end{eqnarray*}
whence
\begin{eqnarray*}
\begin{split}\|A(\Phi &(e^{\alpha z(\theta_t\omega)}u))
-A(\Phi (e^{\alpha
z(\theta_t\omega)}v))\|\\
&\le C(1+e^{\alpha(p-1)
z(\theta_t\omega)}(\|u\|+\|v\|)^{p-1})\|u-v\|\\
&\quad\le C(1+e^{\alpha(p-1) \max_{t\in [0,
T]}|z(\theta_t\omega)|}(\|u\|+\|v\|)^{p-1})\|u-v\|,
\end{split}
\end{eqnarray*}
and
\begin{eqnarray*}
\begin{split}
\||u|^{p-1}u-|v|^{p-1}v\|\le C'(\|u\|+\|v\|)^{p-1}\|u-v\|,
\end{split}
\end{eqnarray*}
which implies that $A(\Phi (e^{\alpha z(\theta_t\omega)}v))$ and
$|v|^{p-1}v$ are Lipschitz in bounded sets of $\ell^2$ with respect
to $v$ uniformly for any $t\in [0, T]$. So by standard arguments,
system \eqref{sys4} possesses a local solution $v(\cdot, \omega,
v_0)\in \mathcal{C}([0, T_{\max}), \ell^2)$, where $[0, T_{\max})$
is the maximal interval of existence of the solution of
\eqref{sys4}. Next, we need to prove that the local solution in fact
a global one. From \eqref{sys4}, it yields that
\begin{eqnarray}\label{est1}
\begin{split}
\frac{d}{dt}\|v(t)\|^2+&2\lambda e^{\alpha
(p-1)z(\theta_t\omega)}\|v\|_{p+1}^{p+1} =2e^{-\alpha
z(\theta_t\omega)}(-A(\Phi (e^{\alpha z(\theta_t\omega)}v)),
v)\\
&+2(\alpha z(\theta_t\omega)-\lambda)\|v\|^2+2e^{-\alpha
z(\theta_t\omega)}(g, v).
\end{split}
\end{eqnarray}
By \eqref{cond1}, we have
\begin{eqnarray}\label{est11}
\begin{split}
&(e^{-\alpha z(\theta_t\omega)}A(\Phi (e^{\alpha
z(\theta_t\omega)}v)), v)=e^{-2\alpha z(\theta_t\omega)}(B(\Phi
(e^{\alpha z(\theta_t\omega)}v)), B(e^{\alpha
z(\theta_t\omega)}v))\\
&~=e^{-2\alpha z(\theta_t\omega)}\sum_{i\in \mathbb{Z}}(\Phi
(e^{\alpha z(\theta_t\omega)}v_{i+1})-\Phi (e^{\alpha
z(\theta_t\omega)}v_i), e^{\alpha
z(\theta_t\omega)}(v_{i+1}-v_i)\\
&\quad\quad\ge ke^{\alpha (p-1)z(\theta_t\omega)}\sum_{i\in
\mathbb{Z}}|v_{i+1}-v_i|^{p+1}-e^{-2\alpha
z(\theta_t\omega)}\sum_{i\in \mathbb{Z}}a_i\\
&\quad\quad\quad\ge-e^{-2\alpha z(\theta_t\omega)}\sum_{i\in
\mathbb{Z}}a_i,
\end{split}
\end{eqnarray}
which implies
\begin{eqnarray}\label{est2}
\begin{split}
\frac{d}{dt}\|v(t)\|^2+&2\lambda e^{\alpha
(p-1)z(\theta_t\omega)}\|v\|_{p+1}^{p+1}+\lambda\|v(t)\|^2\\
&\le (2\alpha z(\theta_t\omega)-\lambda)\|v\|^2
+(\frac{8\|g\|^2}{\lambda}+2\|a\|_1)e^{-2\alpha z(\theta_t\omega)}.
\end{split}
\end{eqnarray}
Due to Gronwall lemma, for $t>0$,
\begin{eqnarray}\label{auxi est}
\begin{split}
&\|v(t)\|^2+\lambda\int_0^te^{-\lambda
s+2\alpha\int_0^sz(\theta_r\omega)dr}\|v(s, \omega, v_0)\|^{2}ds
\\
&\quad+2\lambda\int_0^te^{\alpha (p-1)z(\theta_s\omega)-\lambda
s+2\alpha\int_0^s
z(\theta_r\omega)dr}\|v(s, \omega, v_0)\|_{p+1}^{p+1}ds\\
&\le e^{-\lambda t+2\alpha\int_0^tz(\theta_s\omega)ds}\|v_0\|^2\\
&\quad\quad+(\frac{8\|g\|^2}{\lambda}+2\|a\|_1)e^{-\lambda
t+2\alpha\int_0^tz(\theta_s\omega)ds}\int_0^te^{-2\alpha
z(\theta_s\omega)+\lambda s-2\alpha\int_0^sz(\theta_r\omega)dr}ds.
\end{split}
\end{eqnarray}
Denote
\begin{eqnarray*}
\begin{split}
\eta(\omega)=&(\frac{8\|g\|^2}{\lambda}+2\|a\|_1)\max_{t\in [0,
T]}(e^{-\lambda
t+2\alpha\int_0^tz(\theta_s\omega)ds}\int_0^te^{-2\alpha
z(\theta_s\omega)+\lambda s-2\alpha\int_0^sz(\theta_r\omega)dr}ds)
\end{split}
\end{eqnarray*}
and
\begin{eqnarray*}
\begin{split}
\xi(\omega)=2\alpha\int_0^T|z(\theta_s\omega)|ds,
\end{split}
\end{eqnarray*}
then we have
\begin{eqnarray*}
\begin{split}
\|v(t)\|^2\le& \|v_0\|^2e^{\xi(\omega)}+\eta(\omega),
\end{split}
\end{eqnarray*}
which implies that the solution $v$ is defined in any interval $[0,
T]$.

(2) Let $u_0, v_0\in \ell^2$ and $X(t)=X(t, \omega, u_0), Y(t)=Y(t,
\omega, v_0)$ be two solutions of \eqref{sys4}. Then, denoting
$\Delta(t)=X(t)-Y(t)$, we get
\begin{eqnarray*}
\begin{split}
\frac{d\Delta(t)}{dt}=&-e^{-\alpha z(\theta_t\omega)}(A(\Phi
(e^{\alpha z(\theta_t\omega)}X))-A(\Phi (e^{\alpha
z(\theta_t\omega)}Y)))+(\alpha
z(\theta_t\omega)-\lambda)\Delta(t)\\
&-e^{\alpha(p-1)z(\theta_t\omega)} (|X|^{p-1}X-|Y|^{p-1}Y),
\end{split}
\end{eqnarray*}
and then
\begin{eqnarray*}
\begin{split}
\frac{d}{dt}&\|\Delta(t)\|^2=-2e^{-\alpha z(\theta_t\omega)}(A(\Phi
(e^{\alpha z(\theta_t\omega)}X))-A(\Phi (e^{\alpha
z(\theta_t\omega)}Y)), e^{\alpha
z(\theta_t\omega)}\Delta(t))\\
&\quad+2(\alpha
z(\theta_t\omega)-\lambda)\|\Delta(t)\|^2-e^{\alpha(p-1)z(\theta_t\omega)}
(|X|^{p-1}X-|Y|^{p-1}Y,
\Delta(t))\\
&\quad\quad\quad\quad\le 2(L+\alpha
z(\theta_t\omega))\|\Delta(t)\|^2\le \varrho \|\Delta(t)\|^2,
\end{split}
\end{eqnarray*}
where $\varrho=2(L+\alpha\max_{t\in [0, T]}|z(\theta_t\omega)|)$,
here $L$ denotes the Lipschitz constant of $\Phi$ and the term
$|X|^{p-1}X-|Y|^{p-1}Y$ corresponding to a bounded set where $X$ and
$Y$ belong to. Now by a simply computation, we have
\begin{eqnarray*}
\begin{split}
\sup_{t\in [0, T]}\|X(t)-Y(t)\|^2\le e^{\varrho T}\|u_0-v_0\|^2,
\end{split}
\end{eqnarray*}
which completes the proof.

\end{proof}

\begin{theorem}\label{RDS}
System \eqref{sys4} generates a continuous random dynamical system
$(\varphi(t))_{t\ge 0}$ over $(\Omega, \mathcal{F}, \mathbb{P},
(\theta_t)_{t\in \mathbb{R}})$, where
\begin{eqnarray*}
\begin{split}
\varphi(t, \omega, v_0)=v(t, \omega, v_0)=e^{-\alpha
z(\theta_t\omega)}u(t, \omega, e^{\alpha z(\omega)}v_0)
\end{split}
\end{eqnarray*}
for $v_0\in \ell^2$, $t\ge 0$ and for all $\omega\in \Omega$.
\end{theorem}

\begin{proof}
Actually, that $\varphi$ is a continuous random dynamical system
follows from Theorem \ref{uniquness}. The measurability of $\varphi$
is indicated by the transformation in \eqref{OU1}. The rest of the
proof just follows from the chain rule.
\end{proof}

The random dynamical system $\varphi$ generated by \eqref{sys4} is
conjugated to the one generated by \eqref{sys3} (see \cite{CL}). In
the sequel, we will just consider the random dynamical system
$\varphi$.

\section{Existence of a unique global random attractor}\label{s4}
In this section, we will prove the existence of a random attractor
for the SLDS generated by system \eqref{sys4}. Our main result is
\begin{theorem}\label{random attra.}
The SLDS $\varphi$ generated by system \eqref{sys4} has a unique
global random attractor.
\end{theorem}

To prove Theorem \ref{random attra.} we will use Proposition
\ref{condition}. We first need to prove that there exists an
absorbing set for $\varphi$ in $\ell^2$. Next, we will show the RDS
$\varphi$ is random asymptotically null in the sense of
\eqref{Null}.

\subsection{Existence of an absorbing set}
We need to prove that there exists a closed random tempered set
$\mathcal{K}\in\mathcal{D}(\ell^2)$ of absorption.
\begin{lemma} \label{Absorbing}
There exists a closed random tempered set $\mathcal{K}(\omega)\in $
$\mathcal{D}(\ell^2)$ such that for all $B\in \mathcal{D}(\ell^2)$
and a.e. $\omega \in \Omega$ there exists $T_{B}(\omega)>0$ such
that
$$\varphi(t,\theta_{-t}\omega)B(\theta_{-t}\omega)\subset
\mathcal{K}(\omega) \ \ \mbox{for all} \ \ t\geq T_{B}(\omega).$$
\end{lemma}

\begin{proof}
Let us start with $v(t)=\varphi(t, \omega, v_0)$. Now, by replacing
$\omega$ with $\theta_{-t}\omega$ in \eqref{auxi est}, we obtain
\begin{eqnarray}
\begin{split}
\|&\varphi(t, \theta_{-t}\omega, v_0)
\|^2+\lambda\int_0^te^{-\lambda
s+2\alpha\int_0^sz(\theta_{r-t}\omega)dr}
\|v(s, \theta_{-t}\omega, v_0)\|^{2}ds\\
&~~\quad+2\lambda\int_0^te^{\alpha
(p-1)z(\theta_{s-t}\omega)-\lambda
s+2\alpha\int_0^sz(\theta_{r-t}\omega)dr}
\|v(s, \theta_{-t}\omega, v_0)\|_{p+1}^{p+1}ds\\
&\leq e^{-\lambda t+2\alpha\int_0^t
z(\theta_{s-t}\omega)ds}\|v_0\|^2\\
&~~\quad+(\frac{8\|g\|^2}{\lambda}+2\|a\|_1) \int_0^te^{-2\alpha
z(\theta_{s-t}\omega)+\lambda (s-t)
+2\alpha\int_s^tz(\theta_{r-t}\omega)dr}ds\\
&\leq e^{-\lambda t+2\alpha\int_{-t}^0
z(\theta_{s}\omega)ds}\|v_0\|^2\\
&~~\quad\quad+(\frac{8\|g\|^2}{\lambda}+2\|a\|_1)
\int_{-t}^0e^{-2\alpha z(\theta_{s}\omega)+\lambda s
+2\alpha\int_s^0z(\theta_{r}\omega)dr}ds\\
&\leq e^{-\lambda t+2\alpha\int_{-t}^0
z(\theta_{s}\omega)ds}\|v_0\|^2\\
&~~\quad\quad\quad+(\frac{8\|g\|^2}{\lambda}+2\|a\|_1)
\int_{-\infty}^0e^{-2\alpha z(\theta_{s}\omega)+\lambda s
+2\alpha\int_s^0z(\theta_{r}\omega)dr}ds.\label{4.1}
\end{split}
\end{eqnarray}
Due to \eqref{temperness}, we know that $
\int_{-\infty}^0e^{-2\alpha z(\theta_{s}\omega)+\lambda s
+2\alpha\int_s^0z(\theta_{r}\omega)dr}ds<+\infty.$ Considering for
any $v_0\in B(\theta_{-t}\omega)$, we have
\begin{eqnarray*}
\begin{split}
\|\varphi(t, \theta_{-t}\omega, v_0)\|^2\leq &e^{-\lambda
t+2\int_{-t}^0 z(\theta_{s}\omega)ds}
d(B(\theta_{-t}\omega))^2\\
&~~\quad+(\frac{8\|g\|^2}{\lambda}+2\|a\|_1)
\int_{-\infty}^0e^{-2\alpha z(\theta_{s}\omega)+\lambda s
+2\alpha\int_s^0z(\theta_{r}\omega)dr}ds.
\end{split}
\end{eqnarray*}
Denoting
\begin{eqnarray}
\begin{split}
R^2(\omega)=&1+(\frac{8\|g\|^2}{\lambda}+2\|a\|_1)
\int_{-\infty}^0e^{-2\alpha z(\theta_{s}\omega)+\lambda s
+2\alpha\int_s^0z(\theta_{r}\omega)dr}ds \label{4.2}
\end{split}
\end{eqnarray}
and noticing
\begin{eqnarray}\label{zero}
\lim_{t\rightarrow+\infty}e^{-\lambda t+2\alpha\int_{-t}^0
z(\theta_{s}\omega)ds}d(B(\theta_{-t}\omega))^2=0,
\end{eqnarray}
we conclude that
\begin{eqnarray}
\mathcal{K}(\omega)=\overline{B_{\ell^2}(0, R(\omega))} \label{4.3}
\end{eqnarray}
is an absorbing closed random set. It remains to show that
$\mathcal{K}(\omega)\in \mathcal{D}(\ell^2)$. Indeed, from
Definition \ref{tempered random set} (iv), for all $\gamma>0$, we
get
\begin{eqnarray*}
\begin{split}
&e^{-\gamma t}R^2(\theta_{-t}\omega)\\
&\quad=e^{-\gamma t}+(\frac{8\|g\|^2}{\lambda}+2\|a\|_1)e^{-\gamma
t} \int_{-\infty}^0e^{-2\alpha z(\theta_{s-t}\omega)+\lambda s
+2\alpha\int_s^0z(\theta_{r-t}\omega)dr}ds\\
&\quad\quad=e^{-\gamma
t}+(\frac{8\|g\|^2}{\lambda}+2\|a\|_1)e^{-\gamma t}
\int_{-\infty}^{-t}e^{-2\alpha z(\theta_{s}\omega)+\lambda (s+t)
+2\alpha\int_s^{-t}z(\theta_{r}\omega)dr}ds\rightarrow 0\\
&\quad\quad\quad\quad \mbox{as} \ \ t\rightarrow\infty.
\end{split}
\end{eqnarray*}
Thus, the proof is complete.
\end{proof}

\subsection{Random asymptotic nullity}
In this subsection, the property of random asymptotically null for
the solution $\varphi$ of system \eqref{sys4} will be established.
\begin{lemma} \label{Asymptotic Null}
Let $v_0\in \mathcal{K}(\omega)$ be the absorbing set given by
\eqref{4.3}. Then for every $\epsilon>0$, there exist
$\tilde{T}(\epsilon, \omega, \mathcal{K}(\omega))>0$ and
$\tilde{N}(\epsilon, \omega, \mathcal{K}(\omega))\ge1$, such that
the solution $\varphi$ of problem \eqref{sys4} is random
asymptotically null, that is
\begin{eqnarray*}\sup_{v\in \mathcal{K}(\omega)}\sum_{|i|>\tilde{N}
(\epsilon, \omega, \mathcal{K}(\omega))}| \varphi_i(t,
\theta_{-t}\omega, v(\theta_{-t}\omega)|^2\leq
 \epsilon^2, \ \ \forall
t\geq \tilde{T}(\epsilon, \omega, \mathcal{K}(\omega)).
\end{eqnarray*}
\end{lemma}

\begin{proof}
Choose a smooth cut-off function satisfying $0\leq \rho(s)\leq 1$
for $s\in \mathbb{R^{+}}$ and $\rho(s)=0$ for $0\leq s\leq 1$,
$\rho(s)=1$ for $s\geq 2$. Suppose there exists a constant $c_0$
such that $|\rho'(s)|\leq c_0$ for $s\in \mathbb{R}^+$.

Let $N$ be a fixed integer which will be specified later, and set
$x=(\rho\left(\frac{|i|}{N}\right)\varphi_{i})_{i \in \mathbb{Z}}$.
Then taking the inner product of \eqref{sys4} with $x$ in $\ell^2$,
we obtain
\begin{eqnarray}
\begin{split}
\frac{d}{dt}\sum_{i\in \mathbb{Z}}&\rho\left(\frac{|i|}{N}\right)|
\varphi_{i}|^2+2\lambda e^{\alpha(p-1)z(\theta_t\omega)} \sum_{i\in
\mathbb{Z}} \rho(\frac{|i|}{N})|\varphi_{i}|^{p+1}\\
&=-2(e^{-\alpha z(\theta_t\omega)}(A(\Phi (e^{\alpha
z(\theta_t\omega)}\varphi)), x)\\
&\quad-2(\lambda-\alpha z(\theta_{t}\omega))\sum_{i\in \mathbb{Z}}
\rho\left(\frac{|i|}{N}\right)|\varphi_{i}|^2\\
&\quad\quad+2e^{-\alpha z(\theta_{t}\omega)}\sum_{i\in \mathbb{Z}}
\rho\left(\frac{|i|}{N}\right)g_i\varphi_i. \label{4.4}
\end{split}
\end{eqnarray}
We now estimate terms in \eqref{4.4} one by one. First, we have
\begin{eqnarray*}
\begin{split}
(&e^{-\alpha z(\theta_t\omega)}(A(\Phi (e^{\alpha
z(\theta_t\omega)}\varphi)), x)=e^{-\alpha z(\theta_t\omega)}(B(\Phi
(e^{\alpha
z(\theta_t\omega)}\varphi)), Bx)\\
&\quad=e^{-\alpha z(\theta_t\omega)}\sum_{i\in \mathbb{Z}}(\Phi
(e^{\alpha z(\theta_t\omega)}\varphi_{i+1})-\Phi (e^{\alpha
z(\theta_t\omega)}\varphi_i))\\
&\quad\quad\quad\cdot (\rho\left(\frac{|i+1|}{N}\right)
\varphi_{i+1}-\rho\left(\frac{|i|}{N}\right)\varphi_{i})\\
&\quad\quad=e^{-\alpha z(\theta_t\omega)}\sum_{i\in \mathbb{Z}}(\Phi
(e^{\alpha z(\theta_t\omega)}\varphi_{i+1})-\Phi (e^{\alpha
z(\theta_t\omega)}\varphi_i))\\
&\quad\quad\quad\quad\cdot[(\rho\left(\frac{|i+1|}{N}\right)
-\rho\left(\frac{|i|}{N}\right))\varphi_{i+1}
+\rho\left(\frac{|i|}{N}\right)(\varphi_{i+1}-\varphi_{i})].
\end{split}
\end{eqnarray*}
By \eqref{cond0} and \eqref{cond1}, we obtain
\begin{eqnarray}\label{est12}
\begin{split}
&e^{-\alpha z(\theta_t\omega)}\sum_{i\in
\mathbb{Z}}\rho\left(\frac{|i|}{N}\right)(\Phi (e^{\alpha
z(\theta_t\omega)}\varphi_{i+1})-\Phi (e^{\alpha
z(\theta_t\omega)}\varphi_i))
(\varphi_{i+1}-\varphi_{i})\\
&\quad\ge ke^{\alpha (p-1) z(\theta_t\omega)}\sum_{i\in
\mathbb{Z}}\rho\left(\frac{|i|}{N}\right)
|\varphi_{i+1}-\varphi_{i}|^{p+1} -e^{-2\alpha z(\theta_t\omega)}
\sum_{i\in \mathbb{Z}}\rho\left(\frac{|i|}{N}\right)a_i\\
&\quad\quad\ge -e^{-2\alpha z(\theta_t\omega)}\sum_{|i|\ge N}a_i
\end{split}
\end{eqnarray}
and
\begin{eqnarray}\label{est13}
\begin{split}
&|\sum_{i\in \mathbb{Z}}(\rho\left(\frac{|i+1|}{N}\right)
-\rho\left(\frac{|i|}{N}\right))(\Phi (e^{\alpha
z(\theta_t\omega)}\varphi_{i+1})-\Phi (e^{\alpha
z(\theta_t\omega)}\varphi_i)) \varphi_{i+1}|\\
&\quad\le \frac{c_0}{N}\sum_{i\in \mathbb{Z}}|\Phi (e^{\alpha
z(\theta_t\omega)}\varphi_{i+1})-\Phi (e^{\alpha
z(\theta_t\omega)}\varphi_i)||\varphi_{i+1}|\\
&\quad\quad\le \frac{2c_0c_1}{N}(e^{\alpha
z(\theta_t\omega)}\|\varphi\|^2+e^{\alpha p
z(\theta_t\omega)}\|\varphi\|_{p+1}^{p+1}).
\end{split}
\end{eqnarray}
For the last term in \eqref{4.4},
\begin{eqnarray}
\begin{split}
\sum_{i\in \mathbb{Z}}\rho\left(\frac{|i|}{N}\right)g_i\varphi_i
\leq \frac{\lambda}{2}\sum_{|i|\geq
N}\rho\left(\frac{|i|}{N}\right)|
\varphi_i|^2+\frac{1}{2\lambda}\sum_{|i|\geq N}|g_i|^2. \label{4.6}
\end{split}
\end{eqnarray}
Combining with \eqref{est12}, \eqref{est13} and \eqref{4.6}, we get
\begin{eqnarray*}
&&\frac{d}{dt}\sum_{i\in \mathbb{Z}}\rho\left(\frac{|i|}{N}\right)|
\varphi_{i}|^2+(\lambda-2\alpha z(\theta_{t}\omega))
\sum_{i\in \mathbb{Z}}\rho\left(\frac{|i|}{N}\right)|\varphi_{i}|^2\\
&&\quad\leq \frac{4c_0c_1}{N}(\|\varphi\|^2+e^{\alpha(p-1)
z(\theta_t\omega)}\|\varphi\|_{p+1}^{p+1})+\sum_{|i|\ge N}(2|a_i|
+\frac{1}{\lambda}|g_i|^2)e^{-2\alpha z(\theta_{t}\omega)}.
\end{eqnarray*}
By using Gronwall's inequality for $t>0$ and substituting
$\theta_{-t}\omega$ for $\omega$, it follows that
\begin{eqnarray}\label{4.7}
\begin{split}
&\sum_{i\in \mathbb{Z}}\rho\left(\frac{|i|}{N}\right)|\varphi_{i}
(t,
\theta_{-t}\omega, v_0(\theta_{-t}\omega))|^2\\
\leq&  e^{-\lambda t+2\alpha\int_{0}^t
z(\theta_{s-t}\omega)ds}\|v_0(\theta_{-t}\omega))\|^2\\
&+\frac{4c_0c_1}{N}\int_0^te^{-\lambda
s+2\alpha\int_0^sz(\theta_{r-t}\omega)dr}
\|\varphi(s, \theta_{-t}\omega, v_0)\|^{2}ds\\
&\quad+\frac{4c_0c_1}{N}\int_0^te^{\alpha
(p-1)z(\theta_{s-t}\omega)-\lambda
s+2\alpha\int_0^sz(\theta_{r-t}\omega)dr}
\|\varphi(s, \theta_{-t}\omega, v_0)\|_{p+1}^{p+1}ds\\
&\quad\quad+\sum_{|i|\ge N}(2|a_i| +\frac{1}{\lambda}|g_i|^2)
\int_{0}^t e^{-\lambda(t-s)+2\alpha\int_{s}^t
z(\theta_{r-t}\omega)dr-2\alpha z(\theta_{s-t}\omega)}ds.
\end{split}
\end{eqnarray}
By Lemma \ref{Absorbing} and \eqref{4.1}, there exists
$T_1=T_1(\epsilon, \omega, \mathcal{K}(\omega))>0$ such that for
$t\ge T_1$,
\begin{eqnarray*}
\begin{split}
\frac{4c_0c_1}{N}\int_0^te^{-\lambda
s+2\alpha\int_0^sz(\theta_{r-t}\omega)dr} \|\varphi(s,
\theta_{-t}\omega, v_0)\|^{2}ds \le \frac{4c_0c_1}{\lambda
N}R^2(\omega)
\end{split}
\end{eqnarray*}
and
\begin{eqnarray*}
\begin{split}
\frac{4c_0c_1}{N}\int_0^t&e^{\alpha
(p-1)z(\theta_{s-t}\omega)-\lambda
s+2\alpha\int_0^sz(\theta_{r-t}\omega)dr} \|\varphi(s,
\theta_{-t}\omega, v_0)\|_{p+1}^{p+1}ds\\
&\quad\quad\le \frac{2c_0c_1}{\lambda N}R^2(\omega),
\end{split}
\end{eqnarray*}
where $R(\omega)$ is given by \eqref{4.2}. Since $a\in \ell^1$ and
$g\in \ell^2$, by using \eqref{temperness}, there exist
$\tilde{T}(\epsilon, \omega, \mathcal{K}(\omega))>T_1$ and
$\tilde{N}(\epsilon, \omega, \mathcal{K}(\omega))\ge1$ such that for
$t>\tilde{T}(\epsilon, \omega, \mathcal{K}(\omega))$,
\begin{eqnarray*}
\begin{split}
\sum_{|i|>\tilde{N}(\epsilon, \omega, \mathcal{K}(\omega))}
|\varphi_i(t, \theta_{-t}\omega, v_0(\theta_{-t}\omega)|^2&\leq
\sum_{i\in \mathbb{Z}}\rho\left(\frac{|i|}{N}\right)|\varphi_{i} (t,
\theta_{-t}\omega, v_0(\theta_{-t}\omega))|^2\\
&\quad\quad\le \epsilon^2.
\end{split}
\end{eqnarray*}
The proof is complete.

\end{proof}

Thus, we have proved Theorem \ref{random attra.}.

\section*{Acknowledgements}
The authors would like to thank the former anonymous referees and
the editors of AMC for their helpful comments and suggestions which
largely improve the presentment of the manuscript.

\end{document}